\documentclass[11pt,a4paper]{amsart}

\usepackage[T1]{fontenc}
\usepackage[utf8]{inputenc}
\usepackage{lmodern}
\usepackage{microtype}
\usepackage[a4paper,margin=1.15in]{geometry}
\usepackage{amsmath,amssymb,amsthm,mathtools}
\usepackage{enumitem}
\usepackage[dvipsnames]{xcolor}
\usepackage{aliascnt}
\usepackage{hyperref}
\usepackage[nameinlink,capitalise]{cleveref}

\hypersetup{
  colorlinks=true,
  linkcolor=blue!50!black,
  citecolor=blue!50!black,
  urlcolor=blue!50!black
}

\newtheorem{theorem}{Theorem}[section]

\newaliascnt{lemma}{theorem}
\newtheorem{lemma}[lemma]{Lemma}
\aliascntresetthe{lemma}

\newaliascnt{proposition}{theorem}
\newtheorem{proposition}[proposition]{Proposition}
\aliascntresetthe{proposition}

\newaliascnt{corollary}{theorem}
\newtheorem{corollary}[corollary]{Corollary}
\aliascntresetthe{corollary}

\theoremstyle{definition}
\newaliascnt{definition}{theorem}
\newtheorem{definition}[definition]{Definition}
\aliascntresetthe{definition}

\theoremstyle{remark}
\newaliascnt{remark}{theorem}
\newtheorem{remark}[remark]{Remark}
\aliascntresetthe{remark}

\crefname{theorem}{Theorem}{Theorems}
\Crefname{theorem}{Theorem}{Theorems}
\crefname{lemma}{Lemma}{Lemmas}
\Crefname{lemma}{Lemma}{Lemmas}
\crefname{proposition}{Proposition}{Propositions}
\Crefname{proposition}{Proposition}{Propositions}
\crefname{corollary}{Corollary}{Corollaries}
\Crefname{corollary}{Corollary}{Corollaries}
\crefname{definition}{Definition}{Definitions}
\Crefname{definition}{Definition}{Definitions}
\crefname{remark}{Remark}{Remarks}
\Crefname{remark}{Remark}{Remarks}

\DeclareMathOperator{\conv}{conv}
\DeclareMathOperator{\aff}{aff}
\DeclareMathOperator{\lin}{lin}

\newcommand{\R}{\mathbb{R}}

\newcommand{\VP}{\mathcal P}

\newcommand{\V}{\mathcal V}
\newcommand{\E}{\mathcal E}
\newcommand{\F}{\mathcal F}

\newcommand{\sn}{\mathbb S^2}

\title[The Symmetric Mahler inequality in Dimension Three]{The Symmetric Mahler Inequality in Dimension Three via Admissible Shadow Systems}

\author{Shibing Chen}
\address{School of Mathematical Sciences,
University of Science and Technology of China,
Hefei, Anhui 230026, China}
\email{chenshib@ustc.edu.cn}

\author{Yuanyuan Li}
\address{Institute for Theoretical Sciences,
Westlake University, Hangzhou, 310030, China}
\email{lyyuan@westlake.edu.cn}

\author{Dongmeng Xi}
\address{Department of Mathematics,
Shanghai University,
Shanghai 200444, China}
\email{xi\_dongmeng@shu.edu.cn; dongmeng.xi.math@gmail.com}

\author{Zhe-Feng Xu}
\address{School of Mathematical Sciences,
University of Science and Technology of China,
Hefei, Anhui 230026, China}
\email{xzf1998@mail.ustc.edu.cn}

\date{\today}

\begin{document}

\begin{abstract}
The three-dimensional symmetric Mahler inequality states that, for every
origin-symmetric convex body \(K=-K\subset \mathbb{R}^3\),
\[
    \VP(K)=    |K|\,|K^\circ|\geq \frac{32}{3}.
\]
It was recently proved by Iriyeh--Shibata \cite{IS2020}, and a shorter
proof was later given by
Fradelizi--Hubard--Meyer--Rold\'an-Pensado--Zvavitch \cite{FHMRZ}.
Both proofs combine ingenious equipartition arguments of algebraic-topological
origin with delicate geometric estimates inspired by Meyer's argument for unconditional bodies.

In this paper, we give a new proof of this inequality using a purely geometric approach, based on what we call symmetric admissible
shadow systems. This is a natural extension of the new techniques developed in our proof of the three-dimensional non-symmetric Mahler
conjecture \cite{CLXX-Mahler}. 
\end{abstract}

\maketitle

%\medskip
%\noindent\textbf{2020 Mathematics Subject Classification.} 52A20; 52A40; 52B11; 52B12.

\medskip

\section{Introduction}
The volume product is one of the basic affine invariants of a convex body.
For a convex body \(K\subset \R^n\) containing the origin in its interior, its
polar body is
\[
        K^\circ=\{y\in\R^n:x\cdot y\leq 1\text{ for all }x\in K\}.
\]
For an origin-symmetric convex body \(K=-K\), the symmetric Mahler conjecture
predicts
\begin{equation}
    \label{eq-mahlerns}
        \VP(K)=|K|\,|K^\circ| \geq \frac{4^n}{n!},
\end{equation}
where \(\VP(K)\) denotes the volume product. Equality is expected precisely
for Hanner polytopes, which include affine images of cubes and cross-polytopes
as special cases. Mahler proved the planar case in his 1938 paper \cite{Mahler}. In
dimension three, this becomes
\begin{equation}
    \label{eq-mahler3s}
       \VP(K) \geq  \frac{32}{3},
\end{equation}
for every origin-symmetric convex body \(K\subset\R^3\). Equality is attained
precisely by affine images of the cube and of the cross-polytope.

For general convex bodies, without assuming central symmetry, the Mahler
conjecture predicts
\begin{equation}\label{eq-mahler-nonsym1}
    |K|\,|K^z| \geq \frac{(n+1)^{n+1}}{(n!)^2}
\end{equation}
for every convex body \(K\subset \R^n\) and every
\(z\in \operatorname{int} K\), where
\[
        K^z=z+(K-z)^\circ
\]
denotes the polar body of \(K\) with center \(z\). Equivalently, writing
\[
        \VP(K)= |K|\,|K^{s(K)}|,
\]
where \(s(K)\) denotes the Santal\'o point of \(K\), namely the unique point
\(z\in \operatorname{int} K\) minimizing \(z\mapsto |K^z|\), the conjecture can
be written as
\begin{equation}\label{eq-mahler-nonsym2}
    \VP(K)= |K|\,|K^{s(K)}|
    \geq \frac{(n+1)^{n+1}}{(n!)^2}.
\end{equation}
The constant is expected to be attained by simplices. Since \(s(K)\) minimizes
\(z\mapsto |K^z|\), this formulation is equivalent to the apparently stronger
form \eqref{eq-mahler-nonsym1}. The planar case was also proved by
Mahler  in his 1938 paper   \cite{Mahler}.  

The three-dimensional symmetric Mahler conjecture \eqref{eq-mahler3s} was
proved by Iriyeh--Shibata \cite{IS2020}; a shorter proof, based on an
ingenious equipartition argument, was later given by
Fradelizi--Hubard--Meyer--Rold\'an-Pensado--Zvavitch \cite{FHMRZ}.
Earlier landmarks in the development of Mahler's conjecture include
\cite{SaintRaymond1981,Reisner1986,BourgainMilman1987, GMR1988,Kuperberg2008}.

Recently, the three-dimensional    Mahler conjecture for non-symmetric convex bodies
\eqref{eq-mahler-nonsym1} was proved by the present authors
\cite{CLXX-Mahler}, using the admissible shadow system method.

The shadow system method goes back to Rogers and Shephard \cite{RogersShephard1958}.  For the volume product, Campi and Gronchi \cite{CG} established a shadow-system theorem in the symmetric setting, and Meyer and Reisner \cite{MR} proved a shadow-system theorem for the Santal\'o volume product with a rigidity statement in the equality case.

In \cite{CLXX-Mahler}, the notions of admissible shadow systems and
admissible speeds were introduced, and several new tools were developed. The new admissible deformation allows us to pass from one-point motions to coordinated shadow movements of vertices that preserve the face lattice, and hence to treat general convex polytopes.  

In the present paper, we further develop these ideas in the centrally symmetric setting and apply them to the symmetric Mahler conjecture. 
In the symmetric version, the admissible shadow system method must be rebuilt inside the class of centrally symmetric convex bodies: admissible speeds are required to be symmetric and locally affine on the relevant faces, so that the deformation preserves central symmetry, controls the face lattice, and keeps the volume behavior rigid enough to force only trivial speeds at a minimizer.
Finally, via a symmetric admissible shadow system method, we give a new, purely geometric proof of the three dimensional symmetric inequality \eqref{eq-mahler3s}:

\begin{theorem}[New proof of the three-dimensional symmetric Mahler inequality]\label{thm:main}
Every origin-symmetric convex body $K=-K\subset\R^3$ satisfies
\[
        |K|\,|K^\circ|\geq \frac{32}{3}.
\]
An origin-symmetric polytope attains equality   if and only if it is the affine image of a cube or an octahedron.
\end{theorem}

The paper is organized as follows.  Section \ref{sec:tools} sets up the notations and recalls the basic polarity, compactness, and shadow-system tools. Section \ref{sec:admissible} defines symmetric admissible speeds and proves the local face-lattice stability and affine-volume property.  Section \ref{sec:local-exclusion} combines these deformations with the Meyer--Reisner theorem and with Santal\'o
polarity to show that a minimizer in a bounded symmetric vertex class cannot admit a
non-trivial admissible shadow system, nor can its polar body. Section \ref{sec:combinatorial} proves a three-dimensional Euler-type inequality used in the counting argument. Finally, Section \ref{sec:main-theorem} proves Theorem\ref{thm:main}.

\section{Basic notations and tools}\label{sec:tools}

Throughout, $B^3_2$ denotes the Euclidean unit ball in $\R^3$. By a polytope, we always mean a convex polytope. By $\aff A$, $\lin A$, $\conv A$, we mean the affine hull, linear span, and convex hull of a set $A\subset \R^3$. The dimension of a set is defined as the dimension of its affine hull. We shall also use the standard notions of vertices (or extreme points),
faces, and facets of a convex set; see Schneider \cite{Schneider} for a
general reference. 

For a convex body $K\subset\R^3$ with $0\in\operatorname{int}K$, its polar body is
\[
        K^\circ=\{y\in\R^3:x\cdot y\leq1\text{ for all }x\in K\}.
\]
The bipolar identity gives $(K^\circ)^\circ=K$. We shall also use the Santal\'o polar in the shadow system. If $K\subset\R^3$ is a convex body and $z\in\operatorname{int}K$, define
\[
        K^z=\{y\in\R^3:(y-z)\cdot(x-z)\leq1\text{ for all }x\in K\}.
\]
The Santal\'o point $s(K)$ is the unique point of $\operatorname{int}K$ minimizing $z\mapsto |K^z|$. When $K=-K$, its Santal\'o point is the origin and $K^{s(K)}=K^\circ$.
For a general convex body \(K\), we write
\[
    \VP(K)=|K|\,|K^{s(K)}|.
\]
For an origin-symmetric body $K=-K$, we write
\[
        \VP(K)=|K|\,|K^\circ|.
\]
The volume product $\VP$ is affine invariant; see \cite[Section 1]{FMZ2012}. And polarity preserves the product:
\[
        \VP(K^\circ)=\VP(K).
\]
Indeed, $ \VP(K^\circ)=|K^\circ||(K^\circ)^\circ|=|K^\circ||K|=\VP(K)$.

% $|(TK)^\circ|=|\det T|^{-1}|K^\circ|$ and $(K^\circ)^\circ=K$.

We record some basic properties of volume product and polar bodies.
% \begin{lemma}\label{lem:polar-nonincreasing}
% For every convex body $K\subset\R^3$,
% \[
%         \VP(K^*)\leq \VP(K).
% \]
% \end{lemma}
%  Indeed, 
% $$\VP(K^*)=|K^*||(K^*)^{s(K^*)}|\leq |K^*||(K^{s(K)})^{s(K)}|=|K||K^*|=\VP(K).$$

\begin{lemma}\cite[Theorem 1.3]{BayerLee1993}\label{lem:polarity-face-lattice}
Let $P\subset\R^3$ be a full-dimensional polytope and let $z\in\operatorname{int}P$. Then $P^z$ is a full-dimensional polytope and polarity with center $z,$ and
\[
        V(P^z)=F(P),\qquad F(P^z)=V(P).
\]
\end{lemma}

% The following lemma follows from the affine invariance property of volume product, and the basic fact that the
% number of vertices of the limit polytope is no more than the limsup of the number of vertices of the convergent sequence of polytopes.

\begin{lemma}\label{lem:bounded-class}
Fix $N\geq6$, and let $\mathcal C_N^{\rm s}$ be the class of all origin-symmetric three-dimensional convex polytopes in $\R^3$ with at most $N$ vertices.  Then $\VP$ attains its infimum on $\mathcal C_N^{\rm s}$.
\end{lemma}

\begin{proof}
% Let $(P_m)\subset\mathcal C_N^{\rm s}$ be a minimizing sequence. By affine invariance of $\VP$, we may assume
%  each $P_m$ is in John position after performing an affine transformation. Thus,
% by John's lemma  we have
% \[
%         B^3_2\subset P_m\subset \sqrt 3B^3_2.
% \]
% Write
% \[
%         P_m=\conv\{x_{m,1},\ldots,x_{m,N}\},
% \]
% allowing repetitions when $P_m$ has fewer than $N$ vertices. Passing to a subsequence we may assume $x_{m,i}\to x_i$ for every $i$. Set
% \[
%         P=\conv\{x_1,\ldots,x_N\}.
% \]
% The inclusions pass to the Hausdorff limit, so $B^3_2\subset P$ and $P$ has at most $N$ vertices. 

The proof is similar to that of Lemma 2.3 in \cite{CLXX-Mahler}. Indeed, given a minimizing sequence \((P_j)\subset\mathcal C_N^{\rm s}\), John's
theorem allows us, after suitable linear normalizations, to obtain uniform
inner and outer bounds; hence a subsequence converges to a limit convex body
\(P\). The origin-symmetry of \(P\) follows directly from the fact that \(P\) is the Hausdorff limit of origin-symmetric compact convex sets. \end{proof}

% Because the origin remains uniformly in the interior, central polarity is continuous under this Hausdorff convergence, and so is volume \cite[Lemma 3]{FMZ2012}.  Therefore
% \[
%         \VP(P)=\lim_{m\to\infty}\VP(P_m),
% \]
% so $P$ is a minimizer.

A \emph{shadow system} along a unit direction \(\theta\in\mathbb S^2\), as introduced  by Rogers and Shephard \cite{RogersShephard1958}, is a family 
\[
        L_t=\conv\{x+t\alpha(x)\theta:x\in B\},
\]
where the {\it base set}  $B\subset\R^3$ is bounded, the function $\alpha:B\to\R$ is bounded, and $t$ ranges over an interval. It is non-degenerate if every $L_t$ has non-empty interior.  The following results play a crucial role in our proof. 

\begin{theorem}\cite[Proposition 1]{FMZ2012}\label{thm:shadow-input}
Let $(L_t)_{t\in[-a,a]}$ be a non-degenerate shadow system in $\R^3$ along a direction $\theta\in \sn$. Then
\[
        t\longmapsto |L_t^{s(L_t)}|^{-1}
\]
is convex on $[-a,a]$. If, in addition, $t\mapsto |L_t|$ is affine and $t\mapsto \VP(L_t)$ is constant on $[-a,a]$, then there exist $w\in\R^3$ and $\beta\in\R$ such that, for every $t\in[-a,a]$, one has $L_t=A_t(L_0)$, where $A_t:\mathbb{R}^3 \to \mathbb{R}^3$ is the affine map defined by
\[
         A_t(x)=x+t(w\cdot x+\beta)\theta.
\]
\end{theorem}

\begin{corollary}\cite[Corollary 2(2)]{FMZ2012}\label{cor:one-sided}
Let $(L_t)_{t\in[-a,a]}$ be a non-degenerate shadow system in $\R^3$ along a direction $\theta\in \sn$, and suppose that $t\mapsto |L_t|$ is affine. If
\[
        \VP(L_0)=\min_{t\in[-a,a]}\VP(L_t),
\]
then $\VP(L_t)$ is constant on $[0,a]$ or constant on $[-a,0]$.
\end{corollary}

As observed in \cite{CLXX-Mahler}, Corollary~\ref{cor:one-sided} admits the
following slightly stronger form. 

\begin{lemma}\label{lem:vp-constant-interval}\cite[Lemma 2.6]{CLXX-Mahler}
Under the assumptions of Corollary~\ref{cor:one-sided}, the quantity
\(\VP(L_t)\) is constant on \([-a,a]\).
\end{lemma}

Theorem~\ref{thm:shadow-input} and Lemma~\ref{lem:vp-constant-interval} will
not be invoked explicitly later. They are used only indirectly, through the
auxiliary results in Sections \ref{sec:admissible} and \ref{sec:local-exclusion}, which are obtained from the 
corresponding results of \cite{CLXX-Mahler} together with the new symmetric
definitions introduced here.  

\section{symmetric admissible shadow systems}\label{sec:admissible}

Let $P=-P\subset\R^3$ be a fixed, origin-symmetric and three-dimensional polytope  with $V$ 
distinct vertices  \(x_1,\ldots,x_V\). An
$\alpha=(\alpha_1,\ldots,\alpha_V)\in\mathbb R^V$
will be called a {\it  symmetric  speed vector} if $x_i=-x_j$ implies that $\alpha_i=-\alpha_j.$ Given a direction
\(\theta\in \sn\) and a  symmetric speed vector 
\(\alpha\in\mathbb R^V\), define a shadow system 
\[
        x_i(t)=x_i+t\alpha_i\theta,
        \qquad
        P_t=\conv\{x_i(t):1\leq i\leq V\}, \quad t\in [-c,c],
\]
where $c>0$. Note that $\alpha(-x)=-\alpha(x)$ preserves central symmetry. In the following, for a facet $F$ of $P$, $\lin(F-F)$ denotes the two-dimensional subspace of $\mathbb{R}^3$ parallel to $F.$
(We may use the convention $\conv(\varnothing)=\varnothing$ for completeness.)

\begin{definition}[Symmetric admissible speed]\label{def:admissible}
A   symmetric speed vector $\alpha\in \R^V$ is called \emph{$\theta$-admissible} if the following condition holds for every facet $F$ of $P$.
\begin{itemize}[leftmargin=2.2em]
\item[(1)] If $\theta\notin\lin(F-F)$, then there exists an affine function $\ell_F:\aff(F)\to\R$ such that $ \ell_F(x_i)=\alpha_i$ for every vertex $x_i$ of $F$.
\item[(2)] If $\theta\in\lin(F-F)$, no constraint is imposed on the values $\alpha_i$ on the vertices of $F$.
\end{itemize}
\end{definition}

In the parallel case, each moved vertex of $F$ stays in the original facet plane because all displacements are in the two-dimensional vector space $\lin(F-F)$; this is why no planarity constraint is needed for such a facet.
When the direction $\theta$ is fixed or clear from the context, we simply say that $\alpha$ is {\it symmetric admissible}.

It is simple to see that the set  of all  symmetric $\theta$-admissible speeds forms a linear space,  denoted by $A_\theta^{\rm s}(P)$,
and contains the following symmetric and globally affine speeds as a subspace.

\begin{definition}[Trivial speeds]\label{def:trivial}
A   symmetric speed is called \emph{trivial} if there exists $w\in\R^3$ such that
\[
        \alpha_i=w\cdot x_i \qquad (1\leq i\leq V).
\]
\end{definition}
Such speeds are exactly the infinitesimal speeds on the vertex set produced by affine shears
\[
        A_t(x)=x+t(w\cdot x)\theta.
\]
Every symmetric and globally affine speed is $\theta$-admissible for every direction $\theta$: on a non-parallel facet one restricts the affine function $x\mapsto w\cdot x$ to the facet plane, while on a parallel facet there is no condition. The space of symmetric and globally affine speeds has dimension exactly $3$. Indeed, the space of symmetric affine functions on $\R^3$ has dimension $3$, and evaluation on the vertex set is injective because a full-dimensional polytope has vertices that affinely span $\R^3$.

 \begin{remark}[Non-trivial speeds]\label{rem:nontrivial}
A \emph{non-trivial}  admissible speed means that it is not  globally affine. The Lemma \ref{lem:local-exclusion} shows that, at an interior minimum of $\VP$ along a volume-affine origin-symmetric shadow system, admissibility forces the speed to be symmetric and globally affine. Therefore, any admissible speed lying outside the three-dimensional globally affine subspace gives a  contradiction to local minimality.
\end{remark}

To prove Lemmas \ref{lem:facet-persistence}, \ref{lem:volume-affine} and
Proposition \ref{prop:face-affine-shadow} for shadow systems with symmetric
admissible speeds, we first recall several definitions introduced in our
previous paper \cite{CLXX-Mahler}.

Throughout, we   denote the  vertices, edges, and facets  of the three-dimensional polytope \(P\)  by 
\[
\mathcal V=\{x_1,\ldots,x_V\},\qquad
\mathcal E=\{E_1,\ldots,E_{n_1}\},\qquad
\mathcal F=\{F_1,\ldots,F_{n_2}\}.
\]
Let
\[
\mathcal I_0=\{1,\ldots,V\},\qquad
\mathcal I_1=\{1,\ldots,n_1\},\qquad
\mathcal I_2=\{1,\ldots,n_2\}
\]
be the corresponding sets of labels. 
\begin{definition}[Face lattice] \label{def-fll}
    The {\it face lattice}  (or labeled  face lattice)   of \(P\),
denoted by \(\mathcal S(P)\), is encoded by the index sets
$\mathcal I_0, \mathcal I_1,  \mathcal I_2$
together with the set-valued incidence maps
\[\mathcal I_0
\ \xrightarrow{\ \Phi_{1}\ }\
2^{\mathcal I_1},
\qquad
\mathcal I_1
\ \xrightarrow{\ \Phi_{2}\ }\
2^{\mathcal I_2},
\]
where
\[
\Phi_{1}(i)
=
\{j\in\mathcal I_1: x_i\in E_j\},\qquad \Phi_{2}(j)
=
\{k\in\mathcal I_2: E_j\subset F_k\}.
\]
\end{definition}
The above notations $\Phi_1$ and $\Phi_2$ induce the following set-valued incidence map
\[
\Phi_{0}(k)
=
\left\{
i\in \mathcal I_0:
\text{ there exists } j\in\mathcal I_1
\text{ such that }
j\in\Phi_1(i)
\text{ and }
k\in\Phi_2(j)
\right\}.
\]
Then, the vertex-facet incidence reads
\[
x_i\in F_k
\quad\Longleftrightarrow\quad
i\in \Phi_{0}(k).
\]
\begin{definition}[Persistence of the face lattice]
    We say that a   family   \(\{P_t\}_{t\in [-c,c]}\) of polytopes with $P_0=P$
\emph{preserves the  face lattice} if 
\[
\mathcal S(P_t)=\mathcal S(P) \qquad {\rm for~all~} t.
\]
More precisely, $\mathcal S(P_t)= \mathcal S(P)$ means that the
vertices, edges, and facets of \(P_t\) can be labeled as
\[
\{x_i(t):i\in\mathcal I_0\},\qquad
\{E_j(t):j\in\mathcal I_1\},\qquad
\{F_k(t):k\in\mathcal I_2\},
\]
where \(x_i(0)=x_i\), \(E_j(0)=E_j\), and \(F_k(0)=F_k\), such that
\[
x_i(t)\in E_j(t)
\quad\Longleftrightarrow\quad
x_i\in E_j \quad\Longleftrightarrow\quad j\in \Phi_1(i),
\]
and
\[
E_j(t)\subset F_k(t)
\quad\Longleftrightarrow\quad
E_j\subset F_k \quad\Longleftrightarrow\quad k\in \Phi_2(j),
\]
for all \(i\in\mathcal I_0\), \(j\in\mathcal I_1\), and
\(k\in\mathcal I_2\).  
\end{definition}

The next three results, concerning the persistence of the face lattice and the volume-affine property of admissible shadow systems, are not new: they are either results proved in the previous paper or immediate special cases of them. We record them here in the present notation for convenience.

\begin{lemma}[Persistence of the face lattice]\label{lem:facet-persistence}
Let \(P=-P\subset\mathbb R^3\) be an origin-symmetric convex polytope with vertices
\(x_1,\ldots,x_V\).  Put
\[
        x_i(t)=x_i+t\alpha_i\theta,
        \qquad
        P_t=\operatorname{conv}\{x_i(t):1\le i\le V\}.
\]
Assume that the centrally symmetric speed vector \(\alpha\in \mathbb{R}^V\) is \(\theta\)-admissible.  Then, there exists \(c>0\) such that 
$\{P_t\}_{t\in[-c,c]}$ is a continuous origin-symmetric shadow system that preserves the  face lattice. 
\end{lemma}

\begin{lemma} \label{lem:volume-affine}
If there exists $c>0$ such that the family of shadow system 
\[
        P_t=\conv\{x_i+t\alpha_i\theta: 1\leq i\leq V\}
\]
preserves the face lattice, for all \(t \in [-c,c]\), then
\[
        t\longmapsto |P_t|
\]
is affine on \( [-c,c]\).
\end{lemma}

\begin{proposition}\label{prop:face-affine-shadow}
Let $P=-P\subset\R^3$ be a three-dimensional origin-symmetric convex polytope. If $\alpha$ is a centrally symmetric $\theta$-admissible speed, then there is $c>0$ such that $(P_t)_{t\in[-c,c]}$ is a non-degenerate origin-symmetric shadow system, has the same face lattice as $P$, and satisfies that $t\mapsto |P_t|$ is affine on $[-c,c]$.
\end{proposition}

\section{Minimizers admit only trivial symmetric admissible shadow systems}\label{sec:local-exclusion}
For convenience, we call \(\{P_t\}_{t\in[-c,c]}\) a
\emph{symmetric admissible shadow system} along the direction
\(\theta\in \sn\) if the speed vector
\(\alpha\in\mathbb R^V\) is symmetric and admissible. Such a system, with its underlying direction \(\theta\)
understood, is called non-trivial or trivial, if its speed vector is non-trivial or trivial, respectively.

The main purpose of this section is to establish
Lemma \ref{lem:minimizer-exclusion}; 
for this, we first need the following local version.

\begin{lemma}\label{lem:local-exclusion}
Let
\[
        P_t=\conv\{x_i+t\alpha_i\theta:1\leq i\leq V\},
        \qquad t\in[-c,c],
\]
be a symmetric admissible shadow system, where $c>0$ is sufficiently small so that Proposition \ref{prop:face-affine-shadow} applies.  If
\[
        \VP(P_0)=\min_{t\in[-c,c]}\VP(P_t),
\]
then $\alpha$ is trivial; that is, there exists $w\in \mathbb{R}^3$ such that $$\alpha_i=\alpha(x_i)=w\cdot x_i\qquad (1\leq i\leq V). $$
\end{lemma}

%\begin{proof} By Proposition \ref{prop:face-affine-shadow}, the family \((P_t)_{t\in[-c,c]}\) is a non-degenerate volume-affine shadow system and preserves the face lattice. Since \(P_0\) minimizes the volume product along this interval, the local exclusion lemma for admissible shadow systems implies that the speed is globally affine. Hence there exist \(w\in\mathbb R^3\) and \(\beta\in\mathbb R\) such that \[         \alpha_i=w\cdot x_i+\beta\qquad (1\leq i\leq V).\] Using the symmetry of the labelling, we have \[\alpha_{\bar i}=-\alpha_i,\qquad x_{\bar i}=-x_i.\] Therefore \[       -w\cdot x_i+\beta        =\alpha_{\bar i} =-\alpha_i        =-w\cdot x_i-\beta .\] Thus \(\beta=0\). Consequently \[        \alpha_i=w\cdot x_i\] for all \(i\), as claimed. \end{proof}

\begin{proof}
   The above lemma follows directly from \cite[Lemma~4.1]{CLXX-Mahler}. The only
additional observation needed in the present symmetric setting is that the
admissible speed is symmetric: for every pair of opposite vertices \(x_i\) and
\(-x_i\),
\[
        \alpha(-x_i)=-\alpha(x_i).
\]
%Thus the conclusion of \cite[Lemma~4.1]{CLXX-Mahler} applies in the symmetric
setting and gives the desired statement. 
\end{proof}

Once we have Lemma \ref{lem:local-exclusion}, combined with Proposition \ref{prop:face-affine-shadow} we can obtain that a minimizer of $\mathcal P$ on $\mathcal C^s_N$ can admit only  trivial admissible shadow system.

\begin{lemma}\label{lem:minimizer-exclusion}
Fix $N\geq6$, and let $Q\in\mathcal C_N^{\rm s}$ minimize $\VP$ over $\mathcal C_N^{\rm s}$.  Then neither $Q$ nor its polar $Q^\circ$ admits a non-trivial symmetric admissible shadow system along any direction $\theta\in\sn$.
\end{lemma}

\begin{proof}
First suppose that $Q$ admits a non-trivial symmetric admissible shadow system $Q_t$. By Proposition \ref{prop:face-affine-shadow}, the face lattice is preserved and $Q_t\in\mathcal C_N^{\rm s}$ for all small $|t|$. Since $Q$ minimizes the volume product, we have 
\[
        \VP(Q_t)\geq\VP(Q)=\VP(Q_0).
\]
Thus $Q_0$ is an interior minimum along this shadow system, and Lemma \ref{lem:local-exclusion} forces the speed to be trivial, a contradiction.

Now set $L=Q^\circ$. Suppose that $L$ admits a non-trivial symmetric admissible shadow system $L_t$. Again shrink the interval so that  Proposition \ref{prop:face-affine-shadow} applies for $L_t$, where $|t|$ is sufficiently small. Let
\[
        M_t=L_t^\circ.
\]
By Lemma \ref{lem:polarity-face-lattice}, we have that
\[
        V(M_t)=F(L_t).
\]
Since the face lattice of $L_t$ is preserved, $F(L_t)=F(L)$. Since $L=Q^\circ$ has the face lattice dual to that of $Q$, $F(L)=V(Q)\leq N.$
Therefore, 
\[ V(M_t)=F(L_t)=F(L)=V(Q)\leq N \] 
and hence $M_t\in\mathcal C_N^{\rm s}$ for all small $t$. Since $Q$ minimizes the volume product, 
\[ \VP(M_t)\geq\VP(Q). \]
Since
\[
        \VP(L_t)=\VP(L_t^\circ)=\VP(M_t),
        \qquad
        \VP(L)=\VP(Q^\circ)=\VP(Q),
\]
we have that
\[
        \VP(L_t)\geq\VP(L)=\VP(L_0)
\]
for all small $t$.  Applying Lemma \ref{lem:local-exclusion} to the system \(L_t\), we conclude
that its speed must be trivial, contradicting the assumption that \(L_t\) is non-trivial.
\end{proof}

\section{A symmetric version of Euler-type inequality}\label{sec:combinatorial}

Let $P$ be a three-dimensional origin-symmetric polytope. The terminology used below, including the labeled face lattice, was introduced
in Section \ref{sec:admissible}.  Moreover, let 
\[
         V(P) = |\V(P)|,\qquad  E(P)=|\E(P)|,\qquad  F(P)=|\F(P)|
\]
denote the numbers of vertices, edges, and facets. 

Let $G\in \F(P)$ be a facet. Write
\[
        I(G)=\{i:x_i\in G\},
        \qquad
        m(G)=\#I(G).
\]
For $i\in\mathcal I_0$, let
\[
        d(x_i)=\#\Phi_1(i)
\]
denote the number of edges incident to $x_i$. 

Note that the facets occur in opposite pairs and $m(G)=m(-G)$. For a direction $\theta\in\sn$, define
\begin{equation}\label{eq:Ctheta-definition}
        C_\theta(P)=
        \frac12 \cdot \sum_{\theta\in\lin(G-G)}(m(G)-3),
\end{equation}
where the sum is over all facets parallel to $\theta$. 

The well-known Euler's formula reads 
\[ V(P)+F(P)-E(P)=2.\]
We shall derive the following Euler-type inequality, which relates the
dimension of the linear space of admissible speeds to the quantities
\(V\), \(F\), and \(C_\theta(P)\).

\begin{lemma}\label{lem:dimension-count}
Let $P=-P\subset\R^3$ be an origin-symmetric polytope.  Then
\begin{equation}\label{eq:dimension-estimate}
        \dim A_\theta^{\rm s}(P)
        \geq
        \frac{F(P)-V(P)}2+2+C_\theta(P).
\end{equation}
Consequently, if the right-hand side of \eqref{eq:dimension-estimate} is greater than $3$, then $P$ admits a non-trivial symmetric admissible shadow system.
\end{lemma}
Regarding its proof,
we would like to proceed the proof in a staightforward way, 
while a more explicit method by using the evaluation maps can be seen in \cite{CLXX-Mahler}. 

\begin{proof}
Write \(V=V(P)\), \(E=E(P)\), and \(F=F(P)\). Since the speeds
\(\alpha\in\R^V\) are symmetric (in fact, they are odd vectors), their values are determined by choosing one
vertex from each opposite pair. Thus the number of independent vertex
variables is \(V/2\).

Let
\[
        \mathcal F_\theta
        =
        \{G\in \mathcal F:\theta\notin \lin(G-G)\}
\]
be the set of facets not parallel to \(\theta\). For each facet \(G\in \mathcal F_\theta\), the number of linear restrictions imposed by \(G\) is at
most \(m(G)-3\), since an affine function on the plane containing \(G\) is
determined by its values at three non-collinear vertices. If instead
\[
        \theta\in \lin(G-G),
\]
then \(G\) imposes no constraint by the definition of admissibility. Since opposite facets give the same restrictions after imposing the symmetry of the speed, we have
\[
        \frac V2 - \dim A_{\theta}^{\rm s}(P)
        \leq
        \frac12\sum_{G\in \F_\theta} \bigl(m(G)-3\bigr).
\]
It follows that
\[
\begin{aligned}
        \dim A_{\theta}^{\rm s}(P)
        &\geq
        \frac V2
        -\frac12\sum_{G\in \F_\theta} \bigl(m(G)-3\bigr)  \\
        &=
        \frac V2+\frac32F-\frac12\sum_{G\in \F}m(G)+C_\theta(P) \\
        &=
        \frac V2+\frac32F-E+C_\theta(P).
\end{aligned}
\]
By Euler's formula, the last expression is equal to
\[
        \frac{F-V}{2}+2+C_\theta(P).
\]
Hence \eqref{eq:dimension-estimate} follows.

Note that the trivial speed space
\[
        \{(w\cdot x_1, \ldots , w\cdot x_V) : w\in\R^3\}
\]
is a three-dimensional subspace of \(A_\theta^{\rm s}(P)\). Hence, if the
lower bound in \eqref{eq:dimension-estimate} is greater than \(3\), then
\(A_\theta^{\rm s}(P)\) contains a symmetric admissible speed which is not
globally affine. Therefore it defines a non-trivial symmetric admissible shadow
system
\[
        P_t=\conv\{x_i+t\alpha_i\theta:1\leq i\leq V\}
\]
for sufficiently small \(|t|\).
\end{proof}

The following lemma implies that a three-dimensional origin-symmetric convex polytope with only trivial symmetric admissible shadow systems must be a parallelepiped or an affine image of an octahedron.
\begin{lemma}\label{lem:combinatorial}
Let $P=-P\subset\R^3$ be an origin-symmetric polytope.  If neither $P$ nor $P^\circ$ admits a non-trivial symmetric admissible shadow system, then $P$ is a parallelepiped or an affine image of an octahedron.
\end{lemma}

\begin{proof}
Since $C_\theta(P)\geq 0$ and $P$ admits no non-trivial symmetric admissible shadow system, Lemma \ref{lem:dimension-count} gives
\[
        \frac{F(P)-V(P)}2+2\leq3,
\]
and hence
\begin{equation*}\label{eq:F-minus-V}
        F(P)-V(P)\leq2.
\end{equation*}
Applying the same argument to $P^\circ$, whose numbers of vertices and facets are interchanged by polarity, gives
\begin{equation*}\label{eq:V-minus-F}
        V(P)-F(P)\leq2.
\end{equation*}
Thus
\begin{equation}\label{eq:abs-difference}
        |V(P)-F(P)|\leq2.
\end{equation}
Since $P$ is origin symmetric, $V(P)$ and $F(P)$ are even. Hence only the following three cases may occur.

\smallskip
\noindent\emph{Case 1: $V(P)=F(P)+2$.}
 Then we have
\[
        \frac{F(P^\circ)-V(P^\circ)}2
        =\frac{V(P)-F(P)}2=1.
\]
If some vertex $v\in \mathcal{V}(P)$ had degree $d(v)>3$, then the dual facet $G_v$ of $P^\circ$ would have $m(G_v)=d(v)>3$ vertices by Lemma \ref{lem:polarity-face-lattice}.  Choosing $\theta\in\lin(G_v-G_v)$, we would get
\[
        C_\theta(P^\circ)\geq d(v)-3\geq1.
\]
By Lemma \ref{lem:dimension-count},
\[
        \dim A_\theta^{\rm s}(P^\circ)
        \geq1+2+1=4>3,
\]
contradicting the assumption on $P^\circ$.  Therefore every vertex of $P$ has degree $3$, which implies that
\[
        2E=3V.
\]
Using Euler's formula and $F=V-2$, we have that
 $V=8$ and $F=6$.

An origin symmetric three-dimensional polytope with six facets has three opposite facet pairs, so it must be a parallelepiped after performing an affine transform.

\smallskip
\noindent\emph{Case 2: $F(P)=V(P)+2$.}
Then
\[
V(P^\circ)=F(P)=V(P)+2=F(P^\circ)+2.
\]
Moreover, the assumption is invariant under polarity: neither \(P^\circ\) nor
\((P^\circ)^\circ=P\) admits a non-trivial symmetric admissible shadow system.
Thus Case 1 applied to \(P^\circ\) shows that \(P^\circ\) is a parallelepiped.
Consequently \(P\) is the polar of a parallelepiped, hence an affine image of the
octahedron.

\smallskip
\noindent\emph{Case 3: $V(P)=F(P)$.}
Here Lemma \ref{lem:dimension-count} becomes
\[
        \dim A_\theta^{\rm s}(P)\geq2+C_\theta(P).
\]
We claim that every facet is a triangle or a quadrilateral. Indeed, if $m(G)\geq 5$ for some facet $G$ of $P$, choosing $\theta\in\lin(G-G)$ would give $C_\theta(P)\geq2$, hence a non-trivial speed.

Also, we claim that no two quadrilateral facets share an edge. Since if two quadrilateral facets share an edge, which means they are not opposite facets, choosing $\theta$ parallel to this edge would give $C_\theta(P)\geq2$, again impossible.

Let $p$ be the number of triangular facets and $q$ the number of quadrilateral facets. Since $p+q=F=V$ and $3p+4q=2E,$ using Euler's formula we have that
\[
        p=4,
        \qquad
       q=V-4.
\]
Since no two quadrilateral facets share an edge, every edge of a quadrilateral is adjacent to a triangular facet. By counting the quadrilateral--triangle adjacency edges, we obtain
\[
        4q\leq3p=12.
\]
Since the number $q$ is even, we have that $q\leq2$ and therefore $V=q+4\leq6$. However, an origin-symmetric full-dimensional polytope in $\R^3$ has at least three opposite vertex pairs. Hence $V=6$ and $F=6.$ This can not happen since an origin-symmetric three-dimensional polytope with six vertices must be a affine image of the octahedron and has eight facets, contradicting $F=6$. Thus Case 3 cannot occur.

In conclusion, the only remaining possibilities are the parallelepiped and the affine image of an octahedron.
\end{proof}

\section{Proof of the main result}\label{sec:main-theorem}

We now prove Theorem \ref{thm:main}.

\begin{proof}[Proof of Theorem \ref{thm:main}]
We first prove the result for origin-symmetric polytopes. Fix \(N\geq 6\).
By Lemma~\ref{lem:bounded-class}, there exists \(Q\in\mathcal C_N^{\rm s}\)
minimizing \(\VP\) on \(\mathcal C_N^{\rm s}\). By
Lemmas~\ref{lem:minimizer-exclusion} and \ref{lem:combinatorial}, \(Q\) is a
parallelepiped or an affine  image of the octahedron.

Let \(C=[-1,1]^3\) be the cube. Since \(\VP\) is affine invariant  and \(C^\circ\) is the octahedron, a direct computation gives
\[
        \VP(C)= \VP(C^\circ)=\frac{32}{3}.
\]
Thus the minimum of \(\VP\) on \(\mathcal C_N^{\rm s}\) is \(32/3\). Since
\(N\) was arbitrary, the lower bound holds for every origin-symmetric
three-dimensional polytope.

Now let \(K=-K\subset\R^3\) be an arbitrary origin-symmetric convex body.
Choose a sequence \((P_m)_{m\geq 1}\) of origin-symmetric three-dimensional
polytopes converging to \(K\) in the Hausdorff metric. By the polytope case and
the Hausdorff-continuity of the volume product,
\[
        \VP(K)=\lim_{m\to\infty}\VP(P_m)\geq\frac{32}{3}.
\]
\end{proof}


\begin{thebibliography}{99}
\bibitem{BayerLee1993}
	M.~M. Bayer and C.~W. Lee,
	\emph {Combinatorial aspects of convex polytopes},
	\newblock in Handbook of Convex Geometry, Vols. A--B, P.~M. Gruber and J.~M. Wills, eds.,
	North-Holland, Amsterdam, 1993, pp.~485--534.
    
\bibitem{BourgainMilman1987}
J. Bourgain and V. D. Milman,
\emph{New volume ratio properties for convex symmetric bodies in $\R^n$},
Invent. Math. \textbf{88} (1987), 319--340.
DOI: \href{https://doi.org/10.1007/BF01388911}{10.1007/BF01388911}.

\bibitem{CG}
S. Campi and P. Gronchi,
\emph{On volume product inequalities for convex sets},
Proc. Amer. Math. Soc. \textbf{134} (2006), 2393--2402.

\bibitem{CLXX-Mahler} 
S. Chen, Y. Li, D. Xi, and Z. Xu, The non-symmetric Mahler conjecture in dimension three, 	arXiv:2605.09334 [math.MG].

\bibitem{FHMRZ}
M. Fradelizi, A. Hubard, M. Meyer, E. Rold\'an-Pensado, and A. Zvavitch,
\emph{Equipartitions and Mahler volumes of symmetric convex bodies},
Amer. J. Math. \textbf{144} (2022), 1201--1219.

\bibitem{FMZ2012}
M. Fradelizi, M. Meyer, and A. Zvavitch,
\emph{An application of shadow systems to Mahler's conjecture},
Discrete Comput. Geom. \textbf{48} (2012), 721--734.
DOI: \href{https://doi.org/10.1007/s00454-012-9435-3}{10.1007/s00454-012-9435-3}.

\bibitem{GMR1988}
Y. Gordon, M. Meyer, and S. Reisner,
\emph{Zonoids with minimal volume-product -- a new proof},
Proc. Amer. Math. Soc. \textbf{104} (1988), 273--276.

\bibitem{IS2020}
H. Iriyeh and M. Shibata,
\emph{Symmetric Mahler's conjecture for the volume product in the three-dimensional case},
Duke Math. J. \textbf{169} (2020), no. 6, 1077--1134.
DOI: \href{https://doi.org/10.1215/00127094-2019-0072}{10.1215/00127094-2019-0072}.

\bibitem{Kuperberg2008}
G. Kuperberg,
\emph{From the Mahler conjecture to Gauss linking integrals},
Geom. Funct. Anal. \textbf{18} (2008), no. 3, 870--892.
DOI: \href{https://doi.org/10.1007/s00039-008-0669-4}{10.1007/s00039-008-0669-4}.

\bibitem{Mahler}
K. Mahler, 
\emph{Ein Minimalproblem f{\"u}r konvexe Polygone}, Mathematica B (Zutphen) B7 (1938), 118-127.

\bibitem{Mahler1939}
K. Mahler,
\emph{Ein {\"U}bertragungsprinzip f{\"u}r konvexe K{\"o}rper},
{\v C}asopis pro p{\v e}stov{\'a}n{\'i} matematiky a fysiky
\textbf{68} (1939), 93--102.

\bibitem{MR}
M. Meyer and S. Reisner,
\emph{Shadow systems and volumes of polar convex bodies},
Mathematika \textbf{53} (2006), no. 1, 129--148.
DOI: \href{https://doi.org/10.1112/S0025579300000071}{10.1112/S0025579300000071}.

\bibitem{Reisner1986}
S. Reisner,
\emph{Zonoids with minimal volume-product},
Math. Z. \textbf{192} (1986), 339--346.

\bibitem{RogersShephard1958}
C. A. Rogers and G. C. Shephard,
\emph{Convex bodies associated with a given convex body},
J. London Math. Soc. \textbf{33} (1958), 270--281.

\bibitem{SaintRaymond1981}
J. Saint-Raymond,
\emph{Sur le volume des corps convexes sym\'etriques},
S\'eminaire d'Initiation \`a l'Analyse, 1980--1981, Exp. No. 11, Univ. Paris VI, Paris, 1981.

\bibitem{Schneider}
R. Schneider,
\emph{Convex Bodies: The Brunn--Minkowski Theory},
second expanded edition, Encyclopedia of Mathematics and its Applications, vol. 151,
Cambridge University Press, 2014.
\end{thebibliography}
\end{document}